\documentclass[draft,reqno]{amsart}
\usepackage[american]{babel}
\usepackage{amssymb,euscript}
\numberwithin{equation}{section}
\newtheorem{theorem}{Theorem}[section]
\newtheorem{lemma}[theorem]{Lemma}

\theoremstyle{definition}
\newtheorem{definition}[theorem]{Definition}
\theoremstyle{remark}
\newtheorem{remark}[theorem]{Remark}

\DeclareMathOperator{\diag}{diag}
\DeclareMathOperator{\grad}{grad}

\author[Faminskii]{O. S. Balashov, A. V. Faminskii}
\address{RUDN University, 6 Miklukho-Maklaya St, Moscow, 117198, Russian Federation}
\email{balashovos@s1238.ru, afaminskii@sci.pfu.edu.ru}

\title[Problems for systems odd-order evolution equations]{On direct and inverse problems for odd-order systems of quasilinear evolution equations}

\date{}

\begin{document}

\begin{abstract}
Direct and inverse initial-boundary problems on a bounded interval for systems of quasilinear evolution equations with general nonlinearities are considered. In the case of inverse problems conditions of integral overdetermination are introduced and right-hand sides of equations of special types are chosen as controls. Results on well-posedness of such problems are established. Assumptions on smallness of the input data or smallness of a time interval are required.
\end{abstract}

\maketitle

\section{Introduction. Notation. Description of main results}\label{S1}

Consider the following system of odd-order quasilinear equations
\begin{multline}\label{1.1}
u_t - (-1)^l(a_{2l+1}\partial_x^{2l+1} u + a_{2l} \partial_x^{2l} u) -
\sum\limits_{j=0}^{l-1} (-1)^j \partial_x^j\bigl[a_{2j+1}(t,x) \partial_x^{j+1} u + 
a_{2j}(t,x) \partial_x^{j} u\bigr] \\+
\sum\limits_{j=0}^l (-1)^j \partial_x^j \bigl[g_j(t,x,u,\dots,\partial_x^{l-1} u)\bigr] = f(t,x),\quad l\in \mathbb N,
\end{multline}
posed on an interval $I = (0,R)$ ($R>0$ is arbitrary). Here $u=u(t,x) =(u_1,\dots,u_n)^T$, $n\in \mathbb N$, is an unknown vector-function, $f = (f_1,\dots,f_n)^T$, $g_j = (g_{j1},\dots,g_{jn})^T$ are also vector-functions, $a_{2l+1} = \diag(a_{(2l+1)i})$, $a_{2l} = \diag(a_{(2l)i})$, $i=1,\dots,n$, are constant diagonal $n\times n$ matrices, $a_j(t,x) = \bigl(a_{j i m}(t,x)\bigr)$, $i,m = 1,\dots,n$, for $j= 0,\dots,2l-1$, are $n\times n$ matrices. 

In a rectangle $Q_T = (0,T)\times I$ for certain $T>0$ consider an initial-boundary value problem for system \eqref{1.1} with an initial condition
\begin{equation}\label{1.2}
u(0,x) = u_0 (x),\quad x\in [0,R],
\end{equation}
and boundary conditions
\begin{equation}\label{1.3}
\partial_x^j u(t,0) = \mu_j(t),\ j=0,\dots,l-1, \quad \partial_x^j u(t,R) = \nu_j(t), \ j=0,\dots,l,\quad t\in [0,T],
\end{equation}
where $u_0 = (u_{01},\dots,u_{0n})^T$, $\mu_j = (\mu_{j1},\dots,\mu_{jn})^T$, $\nu_j = (\nu_{j1},\dots,\nu_{jn})^T$.

Besides this direct problem consider the following inverse problem: let for any $i=1,\dots,n$ the function $f_i$ is represented in a form
\begin{equation}\label{1.4}
f_i(t,x) \equiv h_{0i}(t,x) +\sum\limits_{k=1}^{m_i} F_{ki}(t) h_{ki}(t,x)
\end{equation}
for certain non-negative integer number $m_i$ (if $m_i=0$ then $f_i = h_{0i}$), where the functions $h_{ki}$ are given and the functions $F_{ki}$ are unknown. Then problem \eqref{1.1}--\eqref{1.3} is supplemented with overdetermination conditions in an integral form: if $m_i>0$ for certain $i$ then
\begin{equation}\label{1.5}
\int_I u_i(t,x) \omega_{ki}(x)\, dx = \varphi_{ki}(t),\quad t\in [0,T], \quad k=1,\dots,m_i,
\end{equation}
for certain given functions $\omega_{ki}$ and $\varphi_{ki}$. In particular, for certain $i$ the overdetermination conditions on $u_i$ can be absent, but in the case of the inverse problem we always assume that 
\begin{equation}\label{1.6}
M = \sum\limits_{i=1}^n m_i >0.
\end{equation}
Then the aim is to find the functions $F_{ki}$ such that the corresponding solution $u$ to problem \eqref{1.1}--\eqref{1.3} satisfies conditions \eqref{1.5}. 

In the case of a single equation $n=1$ equations of type \eqref{1.1} were considered in \cite{F21} (direct problem)  and \cite{F23} (inverse problems). In particular, in these articles one can found examples of physical models, which can be described by such equations: the Korteweg--de~Vries (KdV) and Kawahara equations with generalizations, the Kortewes--de~Vries--Burgers and Benney-Lin equations, the Kaup--Kupershmidt equation and others. However, besides  the single equations, systems of odd-order quasilinear evolution equations also arise in real physical situations. Among such systems on can mention the Majda--Biello system (see \cite{MB})
$$
\begin{cases}
& u_t + u_{xxx} + vv_x =0,\\
& v_t + \alpha v_{xxx} + (uv)_x =0,\quad \alpha>0,
\end{cases}
$$
and more general systems of KdV-type equations with coupled nonlinearities (\cite{BCG}).

The KdV-type Boussinesq system (\cite{BGK, PR, SX})
$$
\begin{cases}
& u_t + v_x + v_{xxx} + (uv)_x =0,\\
& v_t + u_x + u_{xxx} +vv_x =0
\end{cases}
$$
and the coupled system of two KdV equations, derived in \cite{GG} and  studied in \cite{BBM1, BBM2, BPST, LP, MCW, MP, MMP, MOP, NPR} (also with more general nonlinearities)
$$
\begin{cases}
& u_t + uu_x +u_{xxx} + a_3 v_{xxx} + a_1 vv_x + a_2 (uv)_x =0,\\
& b_1 v_t + r v_x +vv_x +b_2 a_3 u_{xxx} + v_{xxx} +b_2 a_2 uu_x=0,\quad b_1>0, b_2>0,
\end{cases}
$$
are not directly written in form \eqref{1.1}, but can be transformed to it by a linear change of unknown functions (see \cite{BBM1, BGK, PR}).

In paper \cite{F21} the initial-boundary value problem \eqref{1.1}--\eqref{1.3} was considered in the scalar case and the result on global well-posedness in the class of weak solutions under small input data was established. For simplicity it was assumed there that $\mu_j(t) = \nu_j(t) \equiv 0$ for $j\leq l-1$. Note that the general case \eqref{1.3} can be reduced to the homogeneous one by the simple substitution $v(t,x) = u(t,x) - \psi(t,x)$, where the sufficiently smooth function $\psi$ satisfies \eqref{1.3} for $j\leq l-1$, while the form of equation \eqref{1.1} is invariant under the corresponding transformation. 

In the present paper the result on global well-posedness of problem \eqref{1.1}--\eqref{1.3} itself is obtained in the class of weak solutions under small input data. Note that in the aforementioned articles in the case of systems such a problem was not studied. The assumptions on system \eqref{1.1} are similar to the ones in \cite{F21, F23} in the case of single equations.

The significance of integral overdermination conditions in inverse problems is discussed in \cite{POV}. The study of inverse problems for the KdV-type equation with the integral overdetermination was started in \cite{F19-1}. In paper \cite{F23} for problem \eqref{1.1}--\eqref{1.3} in the scalar case two inverse problems with one integral overdetemination condition of \eqref{1.5} type were considered. In the first one the right-hand side of the equation of the type similar to \eqref{1.4} was chosen as the control, in the second one --- the boundary data $\nu_l$. Results on well-posedness either for small input data or small time interval were established. In paper \cite{FM} an initial-boundary value problem on a bounded interval for the higher order nonlinear Schr\"{o}dinger equation
$$
i u_t  + a u_{xx} + i b u_x + i u_{xxx} + \lambda |u|^p u +
i \beta \bigl( |u|^p u\bigr)_x + i \gamma \bigl( |u|^p\bigr)_x u= 0
$$
($u$ is a complex-valued function) with initial and boundary conditions similar to \eqref{1.2}, \eqref{1.3} was considered and three inverse problems were studied. The first two of them were similar to the problems considered in \cite{F23} with similar results. In the third problem two overdetermination conditions of \eqref{1.5} type were introduced and both the right-hand side of the equation and the boundary function were chosen as controls. Results were similar to the first two cases.

Note that the inverse problem with two integral overdetermination conditions for the Korteweg--de~Vries type equation
$$
u_t + u_{xxx} + uu_x +\alpha(t)u = F(t)g(t)
$$
in the periodic case, where the functions $\alpha$ and $F$ were unknown, was considered in \cite{LCL} and the existence and uniqueness results were obtained for a small time interval.

In paper \cite{MOP} an inverse problem on a bounded interval with the terminal overdetermination condition 
$$
u(T,x) = u_T(x)
$$
for a given function $u_T$ (such problems are called controllability ones) was studied for the aforementioned coupled system of two KdV equations. Results on existence of solutions under small input data were established.

In the present paper results on well-posedness of inverse problem \eqref{1.1}--\eqref{1.6} are obtained either for small input data or small time interval. Note that since the amount of integral overdetermination conditions is arbitrary, the result is new even in the case of one equation.

Solutions of the considered problems are constructed in a special function space $u \in \bigl(X(Q_T)\bigr)^n$, where for every $i$
$$
u_i (t,x) \in X(Q_T) = C([0,T]; L_2(I)) \cap L_2(0,T; H^l(I)),
$$
endowed with the norm
$$
\|u\|_{(X(Q_T))^n} = \sum\limits_{i=1}^n\Bigl(\sup\limits_{t\in (0,T)} \|u_i(t,\cdot)\|_{L_2(I)} + \|\partial_x^l u_i\|_{L_2(Q_T)}\Bigr).
$$
For $r>0$ let $\overline X_{rn}(Q_T)$ denote the closed ball $\{u\in\bigl( X(Q_T)\bigr)^n: \|u\|_{(X(Q_T))^n} \leq r\}$.

Introduce the notion of a weak solution of problem \eqref{1.1}--\eqref{1.3}.

\begin{definition}\label{D1.1}
Let $u_0 \in \bigl(L_2(I)\bigr)^n$, $\mu_j, \nu_j \in \bigl(L_2(0,T)\bigr)^n \ \forall j$, $f\in \bigl(L_1(Q_T)\bigr)^n$, $a_j \in \bigl(C(\overline Q_T)\bigr)^{n^2} \ \forall j$. A function $u\in \bigl(X(Q_T)\bigr)^n$ is called a weak solution of problem \eqref{1.1}--\eqref{1.3} if $\partial_x^j u(t,0) \equiv \mu_j(t)$, $\partial_x^j u(t,R) \equiv \nu_j(t)$, $j=0,\dots, l-1$, and for all test functions $\phi(t,x)$, such that $\phi\in \bigl(L_2(0,T;H^{l+1}(I))\bigr)^n$, $\phi_t \in \bigl(L_2(Q_T)\bigr)^n$, $\phi\big|_{t=T} \equiv 0$, $\partial_x^{j} \phi\big|_{x=0} = \partial_x^j \phi\big|_{x=R} \equiv 0$, $j=0,\dots,l-1$, $\partial_x^l \phi\big|_{x=0} \equiv 0$, the functions
$\bigl(g_j(t,x,u,\dots,\partial_x^{l-1}u),\partial_x^j\phi\bigr) \in L_1(Q_T)$, $j=0,\dots,l$, and the following integral identity is verified: 
\begin{multline}\label{1.7}
\iint_{Q_T} \Bigl[ (u,\phi_t) - (a_{2l+1}\partial_x^l u, \partial_x^{l+1}\phi) + 
(a_{2l} \partial_x^l u, \partial_x^l\phi)  \\+
\sum\limits_{j=0}^{l-1} \bigl((a_{2j+1}\partial_x^{j+1}u +a_{2j}\partial_x^j u), \partial_x^j\phi \bigr)  - 
\sum\limits_{j=0}^l \bigl(g_j(t,x,u,\dots,\partial_x^{l-1}u), \partial_x^j\phi\bigr) \\ +
(f,\phi) \Bigr]\,dxdt  +
\int_I (u_0, \phi\big|_{t=0}) \,dx + 
\int_0^T (a_{2l+1}\nu_l, \partial_x^l\phi\big|_{x=R}) \,dt =0,
\end{multline}
where $(\cdot,\cdot)$ denotes the scalar product in $\mathbb R^n$. 
\end{definition}

Let $\widehat f(\xi)\equiv \mathcal F[f](\xi)$ and $\mathcal F^{-1}[f](\xi)$ be
the direct and inverse Fourier transforms of a function $f$, respectively . In particular, for
$f\in \mathcal S(\mathbb R)$
$$
\widehat f(\xi)=\int_{\mathbb R} e^{-i\xi x} f(x)\,dx,\qquad \mathcal F^{-1}[f](x)=\frac 1 {2\pi} \int_{\mathbb R} e^{i\xi x} f(\xi)\,d\xi.
$$
For $s\in \mathbb R$ define the fractional order Sobolev space
$$
H^s(\mathbb R)= \bigl\{f: \mathcal F^{-1}[(1+|\xi|^s)\widehat f(\xi)] \in L_2(\mathbb R)\bigr\}
$$
and for certain $T>0$ let $H^s(0,T)$ be the space of restrictions on $(0,T)$ of functions from $H^s(\mathbb R)$. To describe properties of boundary functions $\mu_j$, $\nu_j$ we use the following function spaces. Let $m=l-1$ or $m=l$,  define 
$$
\bigl(\mathcal B^m(0,T)\bigr)^n=\Bigl(\prod\limits_{j=0}^m H^{(l-j)/(2l+1)}(0,T)\Bigr)^n,
$$
endowed with the natural norm.

The coefficients of the linear part of the system further are always assumed to verify the following conditions:
\begin{equation}\label{1.8}
a_{(2l+1)i}>0, \quad a_{(2l)i}\leq 0, \quad i=1,\dots,n,
\end{equation}
and for any $0 \leq j \leq l-1$, $i, m = 1,\dots n$
\begin{equation}\label{1.9}
\partial_x^k a_{(2j+1) i m} \in C(\overline Q_T),\ k=0,\dots,j+1, \quad 
\partial_x^k a_{(2j) i m} \in C(\overline Q_T),\ k=0,\dots,j.
\end{equation}

Let $y_m = (y_{m1},\dots,y_{mn})$ for $m=0,\dots,l-1$.
The functions $g_j(t,x,y_0,\dots,y_{l-1})$ for any $0\leq j \leq l$ are always subjected to the following assumptions: for $i= 1,\dots,n$
\begin{equation}\label{1.10}
g_{j i}, \grad_{y_k} g_{j i} \in C(\overline{Q}_T\times \mathbb R^{ln}), \ j=0,\dots,l-1, \quad g_{ji}(t,x,0,\dots,0) \equiv 0,
\end{equation}
\begin{multline}\label{1.11}
\bigl| \grad_{y_k} g_{ji}(t,x,y_0,\dots,y_{l-1})\bigr| \leq c \sum\limits_{m=0}^{l-1} 
\bigl(|y_m|^{b_1(j,k,m)} + |y_m|^{b_2(j,k,m)}\bigr), 
k=0,\dots,l-1,\\ \forall (t,x,y_0,\dots,y_{l-1}) \in Q_T\times \mathbb R^{ln},
\end{multline}
where $0<b_1(j,k,m) \leq b_2(j,k,m)$, $|y_m| = (y_m,y_m)^{1/2}$. 

Regarding the functions $\omega_{ki}$ we always need the following conditions:
\begin{equation}\label{1.12}
\omega \in H^{2l+1}(I),\quad \omega^{(m)}(0)=0,\ m=0,\dots,l, \quad \omega^{(m)}(R) =0,\ m=0,\dots, l-1,
\end{equation}
for all $\omega_{ki}$ (where here $\omega$ stands for $\omega_{ki}$).

Now we can pass to the main results and begin with the direct problem.

\begin{theorem}\label{T1.1}
Let the coefficients $a_j$, $j=0,\dots,2l+1$, satisfy conditions \eqref{1.8}, \eqref{1.9}. Let the functions $g_j$ satisfy conditions \eqref{1.10}, \eqref{1.11}, where
\begin{equation}\label{1.13}
b_2(j,k,m) \leq \frac{4l-2j-2k}{2m+1} \quad \forall\ j,k,m.
\end{equation}
Let $u_0 \in \bigl(L_2(I)\bigr)^n$, 
$(\mu_0,\dots,\mu_{l-1}) \in \bigl(\mathcal B^{l-1}(0,T)\bigr)^n$, 
$(\nu_0,\dots,\nu_l) \in \bigl(\mathcal B^l(0,T)\bigr)^n$, $f\in \bigl(L_1(0,T;L_2(I))\bigr)^n$ for an arbitrary $T>0$.
Denote
\begin{multline}\label{1.14}
c_0 = \|u_0\|_{(L_2(I))^n} + \|(\mu_0,\dots,\mu_{l-1})\|_{(\mathcal B^{l-1}(0,T))^n} + 
\|(\nu_0,\dots,\nu_l)\|_{(\mathcal B^l(0,T))^n} \\ +\|f\|_{(L_1(0,T;L_2(I)))^n}.
\end{multline}
Then there exists $\delta>0$ such that under the assumption $c_0 \leq \delta$ there exists a unique weak solution of problem \eqref{1.1}--\eqref{1.3} $u\in \bigl(X(Q_T)\bigr)^n$. Moreover, the map
\begin{equation}\label{1.15}
\bigl(u_0, (\mu_0,\dots,\mu_{l-1}), (\nu_0,\dots,\nu_l), f \bigr) \to u
\end{equation}
is Lipschitz continuous on the closed ball of the radius $\delta$ in the space 
$\bigl(L_2(I)\bigl)^n \times \bigl(\mathcal B^{l-1}(0,T)\bigr)^n \times \bigl(\mathcal B^l(0,T)\bigr)^n \times \bigl(L_1(0,T;L_2(I))\bigr)^n$
into the space $\bigl(X(Q_T)\bigr)^n$.
\end{theorem}

\begin{theorem}\label{T1.2}
Let the hypothesis of Theorem~\ref{T1.1} be satisfied except inequalities \eqref{1.13} which are substituted by the following ones:
\begin{equation}\label{1.16}
b_2(j,k,m) < \frac{4l-2j-2k}{2m+1} \quad \forall\ j,k,m.
\end{equation}
Let $c_0$ is given by formula \eqref{1.14}. Then for the fixed $\delta>0$ there exists $T_0>0$ such that if $c_0 \leq \delta$ and $T\in (0, T_0]$ there exists a unique weak solution of problem \eqref{1.1}--\eqref{1.3} $u\in \bigl(X(Q_T)\bigr)^n$. Moreover, the map \eqref{1.15} is Lipschitz continuous on the closed ball of the radius $\delta$ similarly to Theorem~\ref{T1.1}.
\end{theorem}

For the inverse problem the results are the following.

\begin{theorem}\label{T1.3}
Let the coefficients $a_j$, $j=0,\dots,2l+1$, satisfy conditions \eqref{1.8}, \eqref{1.9} and the functions $g_j$ satisfy conditions \eqref{1.10}, \eqref{1.11}, \eqref{1.13}.
Let $u_0 \in \bigl(L_2(I)\bigr)^n$, 
$(\mu_0,\dots,\mu_{l-1}) \in \bigl(\mathcal B^{l-1}(0,T)\bigr)^n$, 
$(\nu_0,\dots,\nu_l) \in \bigl(\mathcal B^l(0,T)\bigr)^n$, 
$h_0 = (h_{01},\dots,h_{0n})^T \in \bigl(L_1(0,T;L_2(I))\bigr)^n$ for an arbitrary $T>0$.
Assume that condition \eqref{1.6} holds and for any $i=1,\dots,n$, satisfying $m_i>0$, for $k=1,\dots m_i$ the functions $\omega_{ki}$ satisfy condition \eqref{1.12};
$\varphi_{k i} \in W_1^1(0,T)$ and
\begin{equation}\label{1.17}
\varphi_{k i}(0) = \int_I u_{0 i}(x) \omega_{k i}(x) \,dx;
\end{equation}
$h_{ki} \in C([0,T]; L_2(I))$ for $k=1,\dots,m_i$. Let
\begin{equation}\label{1.18}
\psi_{k j i}(t) \equiv \int_I h_{j i}(t,x) \omega_{k i}(x)\, dx,\quad k,j = 1,\dots,m_i,
\end{equation}
and assume that
\begin{equation}\label{1.19}
\Delta_i (t) \equiv \det\bigl(\psi_{k j i}(t)\bigr) \ne 0 \quad \forall \ t\in [0,T].
\end{equation}
Denote
\begin{multline}\label{1.20}
c_0 = \|u_0\|_{(L_2(I))^n} + 
\|(\mu_0,\dots,\mu_{l-1})\|_{(\mathcal B^{l-1}(0,T))^n} + 
\|(\nu_0,\dots,\nu_l)\|_{(\mathcal B^l(0,T))^n} \\+
\|h_0\|_{(L_1(0,T;L_2(I)))^n} +
\sum\limits_{i: m_i>0} \sum\limits_{k=1}^{m_i} \|\varphi_{ki}'\|_{L_1(0,T)}.
\end{multline}
Then there exists $\delta>0$ such that under the assumption $c_0\leq \delta$ there exist  functions $F_{ki}\in L_1(0,T)$, $i: m_i>0$, $k=1,\dots,m_i$, and the corresponding solution of problem \eqref{1.1}--\eqref{1.3} $u\in \bigl(X(Q_T)\bigr)^n$ verifying \eqref{1.5}, where the function $f$ is given by formula \eqref{1.4}. Moreover, there exists $r>0$ such that this solution $u$ is unique in the ball $\overline X_{rn}(Q_T)$ with the corresponding unique functions $F_{ki}\in L_1(0,T)$ and the map
\begin{equation}\label{1.21}
\bigl(u_0, (\mu_0,\dots,\mu_{l-1}), (\nu_0,\dots,\nu_l), h_0, \{\varphi_{ki}'\} \bigr) \to (u, \{F_{ki}\})
\end{equation}
is Lipschitz continuous on the closed ball of the radius $\delta$ in the space $\bigl(L_2(I)\bigl)^n \times \bigl(\mathcal B^{l-1}(0,T)\bigr)^n \times \bigl(\mathcal B^l(0,T)\bigr)^n  \times \bigl(L_1(0,T;L_2(I))\bigr)^n \times  \bigl(L_1(0,T)\bigr)^M$ into the space $\bigl(X(Q_T)\bigr)^n \times \bigl(L_1(0,T)\bigr)^M$.
\end{theorem}

\begin{theorem}\label{T1.4}
Let the hypothesis of Theorem~\ref{T1.3} be satisfied except inequalities \eqref{1.13} which are substituted by inequalities \eqref{1.16}.
Let $c_0$ be given by formula \eqref{1.20}. Then two assertions are valid.

1. For the fixed $\delta>0$ there exists $T_0>0$ such that if $c_0\leq \delta$ and $T\in (0,T_0]$, there exist unique functions $F_{ki}\in L_1(0,T)$, $i: m_i>0$, $k=1,\dots,m_i$, and the corresponding unique solution of problem \eqref{1.1}--\eqref{1.3} $u\in \bigl(X(Q_T)\bigr)^n$ verifying \eqref{1.5}, where the function $f$ is given by formula \eqref{1.4}.

2. For the fixed arbitrary $T>0$ there exists $\delta>0$ such that under the assumption $c_0\leq \delta$ there exist unique functions $F_{ki}\in L_1(0,T)$, $i: m_i>0$, $k=1,\dots,m_i$, and the corresponding unique solution of problem \eqref{1.1}--\eqref{1.3} $u\in \bigl(X(Q_T)\bigr)^n$ verifying \eqref{1.5}, where the function $f$ is given by formula \eqref{1.4}.

Moreover, the map \eqref{1.21} is Lipschitz continuous on the closed ball of the radius $\delta$ similarly to Theorem \ref{T1.3}.
\end{theorem}

\begin{remark}\label{R1.1}
Theorems~\ref{T1.2} and~\ref{T1.4} are verified for the aforementioned Majda--Biello system. In the case of such a system with more general nonlinearities
$$
\begin{cases}
& u_t + u_{xxx} + \bigl(g_1(u,v)\bigr)_x =f_1,\\
& v_t + \alpha v_{xxx} +\bigl(g_2(u,v)\bigr)_x =f_2,\quad \alpha>0,
\end{cases}
$$
Theorems~\ref{T1.1} and~\ref{T1.3} are verified if 
$$
|\partial_{y_k} g_j(y_1,y_2)| \leq c\bigl(|y_1|^{b_1} + |y_2|^{b_1} + |y_1|^{b_2} + |y_2|^{b_2}\bigr),\quad k,j =1,2,
$$
where $0< b_1 \leq b_2 \leq 2$, for example, if $g_1(y_1,y_2) = c y_2^3$, $g_2(y_1,y_2) = c_1 y_1^2 y_2 + c_2 y_1 y_2^2$.
\end{remark}

The paper is organized as follows. Section~\ref{S2} contains certain auxiliary results on the corresponding linear initial-boundary value problem and interpolating inequalities. Section~\ref{S3} is devoted to the direct problem, Section~\ref{S4} --- to the inverse one.

\section{Preliminaries}\label{S2}

Further we use the following interpolating inequality (see, for example, \cite{BIN}): there exists a constant $c=c(R,l,p)$ such that for any $\varphi \in H^l(I)$, integer $m\in [0,l)$ and $p\in [2,+\infty]$
\begin{equation}\label{2.1}
\|\varphi^{(m)}\|_{L_p(I)} \leq c\|\varphi^{(l)}\|_{L_2(I)}^{2s} \|\varphi\|_{L_2(I)}^{1-2s} +c\|\varphi\|_{L_2(I)}, \quad s= s(p,l,m) =\frac{2m+1}{4l} - \frac1{2lp}.
\end{equation}

On the basis of \eqref{2.1} in \cite[Lemma~3.3]{F23} the following inequality was proved: let $j\in [0,l]$, $k,m \in [0,l-1]$, $b\in (0, (4l-2j-2k)/(2m+1)]$, then for any functions $v,w \in X(Q_T)$
\begin{multline}\label{2.2}
\bigl\| |\partial_x^m v|^b \partial_x^k w \bigr\|_{L_{2l/(2l-j)}(0,T;L_2(I))}  \\ \leq c\bigl( T^{((4l-2j-2k) - (2m+1)b)/(4l)}  + T^{(2l-j)/(2l)} \bigr) \|v\|^b_{X(Q_T)} \|w\|_{X(Q_T)}.
\end{multline}

Besides nonlinear system \eqref{1.1} consider its linear analogue
\begin{multline}\label{2.3}
u_t - (-1)^l(a_{2l+1}\partial_x^{2l+1} u + a_{2l} \partial_x^{2l} u) -
\sum\limits_{j=0}^{l-1} (-1)^j \partial_x^j\bigl[a_{2j+1}(t,x) \partial_x^{j+1} u + 
a_{2j}(t,x) \partial_x^{j} u\bigr] \\
 =   f(t,x) +\sum\limits_{j=0}^l  (-1)^j \partial_x^j G_j(t,x),
\end{multline}
$G_j =(G_{j1},\dots,G_{jn})^T$.
The notion of a weak solution to the corresponding initial-boundary value problem is similar to Definition~\ref{D1.1}. In particular, the corresponding integral identity (for the same test functions as in Definition~\ref{D1.1}) is written as follows:
\begin{multline}\label{2.4}
\iint_{Q_T} \Bigl[ (u,\phi_t) - (a_{2l+1 }\partial_x^l u, \partial_x^{l+1}\phi) + 
(a_{2l} \partial_x^l u, \partial_x^l\phi)  \\+
\sum\limits_{j=0}^{l-1} \bigl((a_{2j+1}\partial_x^{j+1}u +a_{2j}\partial_x^j u), \partial_x^j\phi \bigr)  + \bigl(f(t,x),\phi\bigr) +
\sum\limits_{j=0}^l \bigl(G_j(t,x), \partial_x^j\phi\bigr) \Bigr]\,dxdt  \\+
\int_I (u_0, \phi\big|_{t=0}) \,dx + 
\int_0^T (a_{2l+1}\nu_l, \partial_x^l\phi\big|_{x=R}) \,dt =0,
\end{multline}

First consider the case $a_j \equiv 0$ for $j\leq 2l-1$. Then system \eqref{2.3} is obviously splitted into the set of separate equations and we can use the corresponding results from \cite{FL} and \cite{F21} for single equations. 

\begin{lemma}\label{L2.1}
Let the coefficients $a_{2l+1}$ and $a_{2l}$ satisfy condition \eqref{1.8}, $a_j\equiv 0$ for $j\leq 2l-1$, $u_0\in \bigl(L_2(I)\bigr)^n$, $(\mu_0,\dots,\mu_{l-1}) \in \bigl(\mathcal B^{l-1}(0,T)\bigr)^n$, $(\nu_0,\dots,\nu_l) \in \bigl(\mathcal B^l(0,T)\bigr)^n$, $f = G_j \equiv 0$ $\forall j$. Then there exists a unique solution of problem \eqref{2.3}, \eqref{1.2}, \eqref{1.3}
$u\in \bigl(X(Q_T)\bigr)^n$ and for any $t \in (0,T]$
\begin{multline}\label{2.5}
\|u\|_{(X(Q_{t}))^n} \leq c(T)\Bigl[ \|u_0\|_{(L_2(I))^n} + \|(\mu_0,\dots,\mu_{l-1})\|_{(\mathcal B^{l-1}(0,t))^n} \\+ \|(\nu_0,\dots,\nu_l)\|_{(\mathcal B^l(0,t))^n} \Bigr].
\end{multline}
\end{lemma}

\begin{proof}
This assertion succeeds from \cite[Lemma~4.3]{FL}.
\end{proof}

\begin{lemma}\label{L2.2}
Let the coefficients $a_{2l+1}$ and $a_{2l}$ satisfy condition \eqref{1.8}, $a_j\equiv 0$ for $j\leq 2l-1$,  $u_0 \equiv 0$, $\mu_j \equiv 0$ for $j=0,\dots,l-1$, $\nu_j\equiv 0$ for $j=0,\dots,l$, $f\in \bigl(L_1(0,T;L_2(I))\bigr)^n$, $G_j \in \bigl(L_{2l/(2l-j)}(0,T;L_2(I))\bigr)^n$, $j=0,\dots,l$.  Then there exists a unique solution $u\in \bigl(X(Q_T)\bigr)^n$ of problem \eqref{2.3}, \eqref{1.2}, \eqref{1.3} and for any $t\in [0,T]$
\begin{equation}\label{2.6}
\|u\|_{(X(Q_{t}))^n} \leq  c(T)\Bigl[ \|f\|_{(L_1(0,t;L_2(I)))^n}+
\sum\limits_{j=0}^l \|G_j\|_{(L_{2l/(2l-j)}(0,t;L_2(I)))^n} \Bigr];
\end{equation}
moreover, for $i=1,\dots,n$ and $\rho(x) \equiv 1+x$
\begin{multline}\label{2.7}
\int_I u_i^2(t,x)\rho(x)\, dx + \iint_{Q_t} \bigl( (2l+1)a_{(2l+1)i} - 2a_{(2l)i}\rho(x)\bigr)  \bigl(\partial_x^l u_i(\tau,x)\bigr)^2 \, dxd\tau \\ \leq 
2\iint_{Q_t} f_{i} u_i \rho \, dx d\tau+ 2 \sum\limits_{j=0}^l \iint_{Q_t} G_{ji} (\partial_x^j u_i \rho + j \partial_x^{j-1} u_i)\, dx d\tau.
\end{multline}
\end{lemma}

\begin{proof}
This assertion succeeds from \cite[Lemma~4]{F21}.
\end{proof}

\begin{theorem}\label{T2.1}
Let the coefficients $a_j$ satisfy conditions \eqref{1.8}, \eqref{1.9}, $u_0\in \bigl(L_2(I)\bigr)^n$, $(\mu_0,\dots,\mu_{l-1}) \in \bigl(\mathcal B^{l-1}(0,T)\bigr)^n$, 
$(\nu_0,\dots,\nu_l) \in \bigl(\mathcal B^l(0,T)\bigr)^n$, $f\in \bigl(L_1(0,T;L_2(I))\bigr)^n$, $G_j \in \bigl(L_{2l/(2l-j)}(0,T;L_2(I))\bigr)^n$, $j=0,\dots,l$. Then there exists a unique solution $u\in \bigl(X(Q_T)\bigr)^n$ of problem \eqref{2.3}, \eqref{1.2}, \eqref{1.3} and for any $t\in (0,T]$
\begin{multline}\label{2.8}
\|u\|_{(X(Q_{t}))^n} \leq c(T)\Bigl[ \|u_0\|_{(L_2(I))^n} + \|(\mu_0,\dots,\mu_{l-1})\|_{(\mathcal B^{l-1}(0,t))^n} \\+ \|(\nu_0,\dots,\nu_l)\|_{(\mathcal B^l(0,t))^n}  +
\|f\|_{(L_1(0,t;L_2(I)))^n}+
\sum\limits_{j=0}^l \|G_j\|_{(L_{2l/(2l-j)}(0,t;L_2(I)))^n} \Bigr].
\end{multline}
\end{theorem}

\begin{proof}
Denote by $w= (w_1,\dots,w_n)^T$ the solution of problem \eqref{2.3}, \eqref{1.2}, \eqref{1.3} constructed in Lemma~\ref{L2.1}
Let $U(t,x) \equiv u(t,x) - w(t,x)$. Consider an initial-boundary value problem for the function $U$:
\begin{multline}\label{2.9}
U_t - (-1)^l(a_{2l+1}\partial_x^{2l+1} U + a_{2l} \partial_x^{2l}U) -
\sum\limits_{j=0}^{l-1} (-1)^j \partial_x^j\bigl[a_{2j+1}(t,x) \partial_x^{j+1} U + 
a_{2j}(t,x) \partial_x^{j} U\bigr] \\ = f(t,x) +
 \sum\limits_{j=0}^l  (-1)^j \partial_x^j \widetilde G_j(t,x),
\end{multline}
where $\widetilde G_l \equiv G_l$, while 
$\widetilde G_j \equiv G_j + a_{2j+1} \partial_x^{j+1} w + a_{2j} \partial_x^j w$
for $j<l$, 
and zero initial and boundary conditions \eqref{1.2}, \eqref{1.3}.
Note that by virtue of \eqref{2.1} for $m=0$ or $m=1$, $j<l$ and $i=1,\dots,n$
$$
\|\partial_x^{j+m} w_i \|_{L_2(I)} \leq c\|\partial_x^l w_i\|_{L_2(I)}^{(j+m)/l}\|w_i\|_{L_2(I)}^{(l-j-m)/l} +c \|w_i\|_{L_2(I)}.
$$
Therefore, $\widetilde G_j \in \bigl(L_{2l/(2l-j)}(0,T;L_2(I))\bigr)^n$ with
\begin{equation}\label{2.10}
\|\widetilde G_j\|_{(L_{2l/(2l-j)}(0,t;L_2(I)))^n} \leq \|G_j\|_{(L_{2l/(2l-j)}(0,t;L_2(I)))^n} + c(T)\|w\|_{(X(Q_{t}))^n}.
\end{equation}

In order to obtain the solution to the initial-value problem for system \eqref{2.9} we apply the contraction principle and first construct it on a small time interval $[0,t_0]$ as the fixed point  of a map $U = \Lambda V$, where for $V\in \bigl(X(Q_{t_0})\bigr)^n$ the function $U\in \bigl(X(Q_{t_0})\bigr)^n$ is the solution to an initial-boundary value problem for a system
\begin{multline}\label{2.11}
U_t - (-1)^l(a_{2l+1}\partial_x^{2l+1} U + a_{2l} \partial_x^{2l}U) =
\sum\limits_{j=0}^{l-1} (-1)^j \partial_x^j\bigl[a_{2j+1}(t,x) \partial_x^{j+1} V + 
a_{2j}(t,x) \partial_x^{j} V\bigr] \\ +f(t,x) +
 \sum\limits_{j=0}^l  (-1)^j \partial_x^j \widetilde G_j(t,x),
\end{multline}
with zero initial and boundary conditions \eqref{1.2}, \eqref{1.3}. Note that similarly to \eqref{2.10}  the hypothesis of Lemma~\ref{L2.2} is verified and such a map is defined for any $t_0 \in (0,T]$. Moreover, according to \eqref{2.6}
\begin{multline}\label{2.12}
\|U\|_{(X(Q_{t_0}))^n} \leq c(T) \Bigl[ \|f\|_{(L_1(0,t_0;L_2(I)))^n} +
\sum\limits_{j=0}^l \|\widetilde G_j\|_{(L_{2l/(2l-j)}(0,t_0;L_2(I)))^n}  \\+
\sum\limits_{j=0}^{l-1} \bigl(\|\partial_x^{j+1} V\|_{(L_{2l/(2l-j)}(0,t_0;L_2(I)))^n} + (\|\partial_x^{j} V\|_{(L_{2l/(2l-j)}(0,t_0;L_2(I)))^n}\bigr) \Bigr].
\end{multline}
By virtue of \eqref{2.1} if $j+m \leq 2l-1$ for $i=1,\dots,n$
\begin{multline}\label{2.13}
\|\partial_x^m V_i\|_{L_{2l/(2l-j)}(0,t_0;L_2(I))} \\\leq c\Bigl(\int_0^{t_0} \bigl( \|\partial_x^l V_i\|_{L_2(I)}^{2m/(2l-j)} \|V_i\|_{L_2(I)}^{2(l-m)/(2l-j)} + 
\|V_i\|_{L_2(I)}^{2l/(2l-j)} \bigr)\,dt \Bigr)^{(2l-j)/(2l)} \\ \leq
c t_0^{(2l-j-m)/(2l)} \|V_i\|_{C([0,t_0];L_2(I))}^{(l-m)/l} 
\|\partial_x^l V_i\|_{L_2(Q_{t_0})}^{m/l} + c t_0^{(2l-j)/(2l)}  \|V_i\|_{C([0,t_0];L_2(I))} \\ 
\leq c(T) t_0^{1/(2l)} \|V_i\|_{X(Q_{t_0})}.
\end{multline}
Therefore, it follows from \eqref{2.12} that
\begin{multline}\label{2.14}
\|U\|_{(X(Q_{t_0}))^n} \leq c(T) \Bigl[\|f\|_{(L_1(0,t_0;L_2(I)))^n}+ \sum\limits_{j=0}^l \|\widetilde G_j\|_{(L_{2l/(2l-j)}(0,t_0;L_2(I)))^n} \\ + t_0^{1/(2l)} \|V\|_{(X(Q_{t_0}))^n} \Bigr].
\end{multline}
Similarly to \eqref{2.14} for $\widetilde V \in \bigl(X(Q_{t_0})\bigr)^n$, $\widetilde U = \Lambda \widetilde V$
\begin{equation}\label{2.15}
\|U - \widetilde U\|_{(X(Q_{t_0}))^n} \leq c(T) t_0^{1/(2l)} \|V - \widetilde V\|_{(X(Q_{t_0}))^n}.
\end{equation}
Inequalities \eqref{2.14}, \eqref{2.15} provide existence of the unique solution $U\in \bigl(X(Q_{t_0})\bigr)^n$ to the considered problem if, for example, $c(T) t_0^{1/(2l)} \leq 1/2$. Then since the value of $t_0$ depends only on $T$ step by step this solution can be extended to the whole time segment $[0,T]$, moreover,
\begin{equation}\label{2.16}
\|U\|_{(X(Q_{t}))^n} \leq c(T) \Bigl[\|f\|_{(L_1(0,t;L_2(I)))^n}\sum\limits_{j=0}^l \|\widetilde G_j\|_{(L_{2l/(2l-j)}(0,t;L_2(I)))^n}\Bigr].
\end{equation}
Combining \eqref{2.5} (applied to the function $w$), \eqref{2.10} and \eqref{2.16}, for $u \equiv U + w$ we complete the proof.
\end{proof}

Introduce certain additional notation. Let 
$$
u = S(u_0, (\mu_0,\dots,\mu_{l-1}), (\nu_0,\dots,\nu_l), f, (G_0,\dots,G_l))
$$
be the solution of problem \eqref{2.3}, \eqref{1.2}, \eqref{1.3} from the space $\bigl(X(Q_T)\bigr)^n$ under the hypothesis of Theorem~\ref{T2.1}. Define also 
$$
W = (u_0, (\mu_0,\dots,\mu_{l-1}),(\nu_0,\dots,\nu_l)), 
$$
$$
\widetilde S W = S(W, 0,(0,\dots,0)),\quad \widetilde S: \bigl(L_2(I)\times \mathcal B^{l-1}(0,T) \times \mathcal B^{l}(0,T)\bigr)^n \to \bigl(X(Q_T)\bigr)^n,
$$
\begin{multline*}
S_0 f = S(0,( 0,\dots,0), (0,\dots,0), f, (0,\dots,0)), \\S_0: \bigl(L_{1}(0,T;L_2(I))\bigr)^n \to \bigl(X(Q_T)\bigr)^n,
\end{multline*}
\begin{multline*}
\widetilde S_j G_j = S(0,( 0,\dots,0), (0,\dots,0), 0, (0,\dots,G_j,\dots,0)), \\S_j: \bigl(L_{2l/(2l-j)}(0,T;L_2(I))\bigr)^n \to \bigl(X(Q_T)\bigr)^n,\quad j=0,\dots,l.
\end{multline*}

Let $\widetilde W_1^1 (0,T) = \{\varphi \in W_1^1(0,T): \varphi(0)=0\}$. Obviously, the equivalent norm in this space is $\|\varphi'\|_{L_1(0,T)}$.

Let a function $\omega \in C(\overline{I})$. On the space of functions $u(t,x)$, lying in $L_1(I)$ for all $t\in [0,T]$, define a linear operator $Q(\omega)$ by a formula $(Q(\omega)u)(t) = q(t;u,\omega)$, where
$$
q(t;u,\omega)\equiv \int_I u(t,x) \omega(x) \,dx,\quad t\in [0,T].
$$

\begin{lemma}\label{L2.3}
Let the hypothesis of Theorem~\ref{T2.1} be verified. Let the function $\omega$ satisfy condition \eqref{1.12}. Then for the function $u = (u_1\dots,u_n)^T =  S(u_0, (\mu_0,\dots,\mu_{l-1}), (\nu_0,\dots,\nu_l), f, (G_0,\dots,G_l))$ the corresponding function $q(\cdot;u_i,\omega)= Q(\omega)u_i \in W_1^1(0,T)$, i=1,\dots,n, and for a.e. $t\in (0,T)$
\begin{multline}\label{2.17}
q'(t;u_i,\omega) = r(t;u_i,\omega) \equiv \nu_{li}(t) a_{(2l+1)i} \omega^{(l)}(R) \\+
\sum\limits_{k=0}^{l-1} (-1)^{l+k}\left[ \nu_{ki}(t) \left( a_{(2l+1)i}\omega^{(2l-k)}(R)  -a_{(2l)i} \omega^{(2l-k-1)}(R) \right)\right. \\-
\left.\mu_{ki}(t) \left( a_{(2l+1)i}\omega^{(2l-k)}(0) - a_{(2l)i} \omega^{(2l-k-1)}(0)\right) \right]  \\ +
\sum\limits_{m=1}^n \sum\limits_{j=0}^{l-1} \sum\limits_{k=0}^{j-1} (-1)^{j+k}\left[ \nu_{km}(t) \left( (a_{(2j+1)im} \omega^{(j)})^{(j-k)}(R) - 
(a_{(2j)im} \omega^{(j)})^{(j-k-1)}(R) \right)\right. \\  -
\left. \mu_{km}(t) \left( (a_{(2j+1)im} \omega^{(j)})^{(j-k)}(0) - (a_{(2j)im} \omega^{(j)})^{(j-k-1)}(0) \right) \right] \\+
(-1)^{l+1} \int_I u_i(t,x) \left(a_{(2l+1)i}\omega^{(2l+1)} - 
a_{(2l)i} \omega^{(2l)} \right) \,dx \\+
\sum\limits_{m=1}^n \sum\limits_{j=0}^{l-1} (-1)^{j+1} \int_I u_m(t,x) \left[ (a_{(2j+1)im} \omega^{(j)})^{(j+1)} - (a_{(2j)im} \omega^{(j)})^{(j)} \right] \,dx \\+
\int_I f_i(t,x) \omega \,dx +
\sum\limits_{j=0}^l \int_I G_{ji}(t,x) \omega^{(j)} \,dx,
\end{multline}
\begin{multline}\label{2.18}
\|q'(\cdot;u_i,\omega)\|_{L_1(0,T)} \leq c(T) \Bigl[ \|u_0\|_{(L_2(I))^n} + \|(\mu_0,\dots,\mu_{l-1})\|_{(\mathcal B^{l-1}(0,T))^n} \\ + \|(\nu_0,\dots,\nu_l)\|_{(\mathcal B^l(0,T))^n}  + \|f\|_{(L_1(0,T;L_2(I)))^n} +
\sum\limits_{j=0}^l \bigl(\|G_j\|_{(L_{2l/(2l-j)}(0,T;L_2(I)))^n} \bigr)\Bigr],
\end{multline}
where the constant $c$ does not decrease in $T$.
\end{lemma}

\begin{proof}
For an arbitrary function $\psi \in C_0^\infty(0,T)$ let $\phi_i(t,x) \equiv \psi(t)\omega(x)$ for certain $i$,  $\phi_m(x) \equiv 0$ when $m\ne i$. This function $\phi$ satisfies the assumption on a test function from Definition~\ref{D1.1} an then equality \eqref{2.4} after integration by parts yields that
\begin{equation}\label{2.19}
\int_0^T \psi'(t) q(t;u_i,\omega) \, dt = -\int_0^T \psi(t) r(t;u_i,\omega) \,dt.
\end{equation}
Since $r\in L_1(0,T)$ it follows from \eqref{2.19} that there exists the Sobolev derivative $q'(t;u_i,\omega) = r(t;u_i,\omega) \in L_1(0,T)$ and
\begin{multline*}
\|q'\|_{L_1(0,T)} \leq c\Bigl[ \sum\limits_{j=0}^{l-1} \|\mu_j\|_{(L_1(0,T))^n} + \sum\limits_{j=0}^l \|\nu_j\|_{(L_1(0,T))^n} + \|f\|_{(L_1(0,T;L_2(I)))^n} \\+
\sum\limits_{j=0}^l \|G_j\|_{(L_1(0,T;L_1(I)))^n} + \|u\|_{(L_1(0,T;L_2(I)))^n} \Bigr].
\end{multline*}
Since $\|u\|_{(L_1(0,T;L_2(I))^n} \leq T \|u\|_{(C([0,T];L_2(I)))^n} \leq T \|u\|_{(X(Q_T))^n}$, application of inequality \eqref{2.8} finishes the proof.
\end{proof}

\section{The direct problem}\label{S3}

\begin{proof}[Proof of the existence part of Theorem~\ref{T1.1}] 
On the space $\bigl(X(Q_T)\bigr)^n$ consider a map $\Theta$
\begin{equation}\label{3.1}
u = \Theta v \equiv \widetilde S W + S_0 f - \sum\limits_{j=0}^l \widetilde S_j g_j(t,x,v,\dots,\partial_x^{l-1} v).
\end{equation}
Note that according to conditions \eqref{1.10}, \eqref{1.11} for $i=1,\dots,n$
\begin{equation}\label{3.2}
|g_{ji}(t,x,v,\dots,\partial_x^{l-1} v)| \leq c \sum\limits_{k=0}^{l-1} \sum\limits_{m=0}^{l-1} \bigl(|\partial_x^m v |^{b_1(j,k,m)} + |\partial_x^m v |^{b_2(j,k,m)} \bigr)  |\partial_x^k v|
\end{equation}
In particular, conditions \eqref{1.13} and inequality \eqref{2.2} yield that $g_{ji}(t,x,v,\dots,\partial_x^{l-1} v) \in L_{2l/(2l-j)}(0,T;L_2(I))$, moreover,
\begin{multline}\label{3.3}
\|g_{j}(t,x,v,\dots,\partial_x^{l-1} v)\|_{(L_{2l/(2l-j)}(0,T;L_2(I)))^n}  \\ \leq 
c \sum\limits_{k=0}^{l-1} \sum\limits_{m=0}^{l-1} \sum\limits_{i=1}^2 (T^{((4l-2j-2k)-(2m+1)b_i(j,k,m))/(4l)} +T^{(2l-j)/(2l)}) \|v\|_{(X(Q_T))^n}^{b_i(j,k,m)+1}.
\end{multline}
In particular, Theorem \ref{2.1} ensures that the map $\Theta$ exists. Let 
\begin{equation}\label{3.4}
b_1 = \min\limits_{j,k,m}(b_1(j,k,m)), \quad b_2 = \max\limits_{j,k,m}(b_2(j,k,m)), \quad 0< b_1 \leq b_2,
\end{equation}
then it follows from \eqref{3.3} that
\begin{equation}\label{3.5}
\|g_j(t,x,v,\dots,\partial_x^{l-1} v)\|_{(L_{2l/(2l-j)}(0,T;L_2(I)))^n} \leq c(T)\left( \|v\|_{(X(Q_T))^n}^{b_1+1} + \|v\|_{(X(Q_T))^n}^{b_2+1}\right).
\end{equation}
therefore, inequality \eqref{2.8} implies that
\begin{equation}\label{3.6}
\|\Theta v\|_{(X(Q_T))^n} \leq c(T)c_0 + c(T)\left( \|v\|_{(X(Q_T))^n}^{b_1+1} + \|v\|_{(X(Q_T))^n}^{b_2+1}\right).
\end{equation}

Next, for any functions $v_1, v_2 \in \bigl(X(Q_T)\bigr)^n$
\begin{multline}\label{3.7}
|g_{ji}(t,x,v_1,\dots,\partial_x^{l-1} v_1) - g_{ji}(t,x,v_2,\dots,\partial_x^{l-1} v_2)|  
\\ \leq c\sum\limits_{k=0}^{l-1} \sum\limits_{m=0}^{l-1} \left(|\partial_x^m v_1 |^{b_1(j,k,m)} + |\partial_x^m v_2 |^{b_1(j,k,m)} + |\partial_x^m v_1 |^{b_2(j,k,m)}+
|\partial_x^m v_2 |^{b_2(j,k,m)} \right) \\ \times |\partial_x^k (v_1-v_2)|,
\end{multline}
therefore, similarly to \eqref{3.5}
\begin{multline}\label{3.8}
\|g_j(t,x,v_1,\dots,\partial_x^{l-1} v_1) - g_j(t,x,v_2,\dots,\partial_x^{l-1} v_2)\|_{(L_{2l/(2l-j)}(0,T;L_2(I)))^n} \\ \leq 
c(T)\left( \|v_1\|_{(X(Q_T))^n}^{b_1} + \|v_2\|_{(X(Q_T))^n}^{b_1} + \|v_1\|_{(X(Q_T))^n}^{b_2}  + \|v_2\|_{(X(Q_T))^n}^{b_2}\right) \\ \times \|v_1-v_2\|_{(X(Q_T))^n}.
\end{multline}

and similarly to \eqref{3.6}
\begin{multline}\label{3.9}
\|\Theta v_1 - \Theta v_2\|_{(X(Q_T))^n}  \\ \leq 
c(T)\left( \|v_1\|_{(X(Q_T))^n}^{b_1} + \|v_2\|_{(X(Q_T))^n}^{b_1} + \|v_1\|_{(X(Q_T))^n}^{b_2} + \|v_2\|_{(X(Q_T))^n}^{b_2}\right) \\ \times \|v_1-v_2\|_{(X(Q_T))^n}.
\end{multline}

Now choose $r>0$ such that
\begin{equation}\label{3.10}
r^{b_1} + r^{b_2} \leq \frac1{4c(T)}
\end{equation}
and then $\delta>0$ such that
\begin{equation}\label{3.11}
\delta \leq \frac{r}{2c(T)}.
\end{equation}
Then it follows from \eqref{3.6} and \eqref{3.9} that on the ball $\overline X_{rn}(Q_T)$ the map $\Theta$ is a contraction. Its unique fixed point $u\in \bigl(X(Q_T)\bigr)^n$ is the desired solution. Moreover,
\begin{equation}\label{3.12}
\|u\|_{(X(Q_T))^n} \leq c(c_0).
\end{equation}
\end{proof}

Note that the above arguments ensures uniqueness only in a certain ball. In order to establish uniqueness and continuous dependence in the whole space we apply another approach. Then the rest part of Theorem~\ref{T1.1} succeeds from \eqref{3.12} and the theorem below.

\begin{theorem}\label{T3.1}
Let the assumptions on the functions $a_j$ and $g_j$ from the hypothesis of Theorem~\ref{T1.1} be satisfied. Let $u_0, \widetilde u_0 \in \bigl(L_2(I)\bigr)^n$, 
$(\mu_0,\dots, \mu_{l-1}), (\widetilde\mu_0,\dots \widetilde\mu_{l-1}) \in \bigl(\mathcal B^{l-1}(0,T)\bigr)^n$, 
$(\nu_0,\dots,\nu_l), (\widetilde\nu_0,\dots,\widetilde\nu_l)  \in \bigl(\mathcal B^l(0,T)\bigr)^n$, 
$f, \widetilde f\in \bigl(L_1(0,T;L_2(I))\bigr)^n$ and let $u, \widetilde u$ be two weak solutions to corresponding problems \eqref{1.1}--\eqref{1.3} in the space $\bigl(X(Q_T)\bigr)^n$ with
$\|u\|_{(X(Q_T))^n}, \|\widetilde u\|_{(X(Q_T))^n} \leq K$ for certain positive $K$. Then there exists a positive constant $c= c(T,K)$ such that
\begin{multline}\label{3.13}
\|u - \widetilde u\|_{(X(Q_T))^n} \leq c \bigl( \|u_0 - \widetilde u_0\|_{(L_2(I))^n} +
\|(\mu_0 - \widetilde\mu_0,\dots,\mu_{l-1} - \widetilde\mu_{l-1})\|_{(\mathcal B^{l-1}(0,T))^n} \\ +
\|(\nu_0 - \widetilde\nu_0,\dots,\nu_{l} - \widetilde\nu_{l})\|_{(\mathcal B^{l}(0,T))^n} +
\|f - \widetilde f\|_{(L_1(0,T;L_2(I)))^n} \bigr).
\end{multline}
\end{theorem}

\begin{proof}
Let $w \in \bigl(X(Q_T)\bigr)^n$ be a solution to a linear problem
\begin{equation}\label{3.14}
w_t - (-1)^l( a_{2l+1} \partial_x^{2l+1} w + a_{2l} \partial_x^{2l} w) =0,
\end{equation}
\begin{equation}\label{3.15}
w(0,x) = u_0(x) - \widetilde u_0(x),
\end{equation}
\begin{equation}\label{3.16}
\partial_x^j w(t,0) = \mu_j(t) - \widetilde\mu_j(t),\ j=0,\dots,l-1,\quad
\partial_x^j w(t,R) = \nu_j(t) - \widetilde\nu_j(t),\ j=0,\dots,l.
\end{equation}
Lemma~\ref{L2.1} ensures that such a function exists and according to \eqref{2.5}
\begin{multline}\label{3.17}
\|w\|_{(X(Q_T))^n} \leq c(T) \bigl( \|u_0 - \widetilde u_0\|_{(L_2(I))^n} +
\|(\mu_0 - \widetilde\mu_0,\dots,\mu_{l-1} - \widetilde\mu_{l-1})\|_{(\mathcal B^{l-1}(0,T))^n} \\ +
\|(\nu_0 - \widetilde\nu_0,\dots,\nu_{l} - \widetilde\nu_{l})\|_{(\mathcal B^{l}(0,T))^n} \bigr).
\end{multline}
Let $v(t,x) \equiv u(t,x) - \widetilde u(t,x) - w(t,x)$, Then $v \in \bigl(X(Q_T)\bigr)^n$ is a solution to an initial-boundary problem in $Q_T$ for a system
\begin{multline}\label{3.18}
v_t - (-1)^l(a_{2l+1}\partial_x^{2l+1} v + a_{2l} \partial_x^{2l} v) = (f - \widetilde f) \\+ 
\sum\limits_{j=0}^{l-1} (-1)^j \partial_x^j\bigl[a_{2j+1}(t,x) \partial_x^{j+1} 
(u - \widetilde u) + a_{2j}(t,x) \partial_x^{j} (u - \widetilde u)\bigr] \\-
\sum\limits_{j=0}^l (-1)^j \partial_x^j \bigl[g_j(t,x,u,\dots,\partial_x^{l-1} u) - 
g_j(t,x,\widetilde u,\dots,\partial_x^{l-1} \widetilde u)\bigr]
\end{multline}
with zero initial and boundary conditions of \eqref{1.2}, \eqref{1.3} type.
Similarly to \eqref{2.11}--\eqref{2.13}
$a_{2j+1}(t,x) \partial_x^{j+1} u + a_{2j}(t,x) \partial_x^{j} u \in 
\bigl(L_{2l/(2l-j)}(0,T;L_2(I))\bigr)^n$,
similarly to \eqref{3.2}, \eqref{3.3}
$g_j(t,x,u,\dots,\partial_x^{l-1} u) \in \bigl(L_{2l/(2l-j)}(0,T;L_2(I))\bigr)^n$. The same properties hold in the case of the function $\widetilde u$. Therefore, the hypothesis of Lemma~\ref{L2.2} is satisfied and for $i=1,\dots,n$ according to \eqref{2.7}
\begin{multline}\label{3.19}
\int_I v_i^2(t,x)\rho\, dx + \iint_{Q_t} \bigl( (2l+1)a_{(2l+1)i} - 2a_{(2l)i}\rho\bigr)  \bigl(\partial_x^l v_i(\tau,x)\bigr)^2 \, dxd\tau \\ \leq 
2 \iint_{Q_t} (f_i - \widetilde f_i) v_i \rho \,dx d\tau \\+
2 \sum\limits_{m=1}^n \sum\limits_{j=0}^{l-1} \iint_{Q_t} 
\Bigl( a_{(2j+1)im}(t,x) \partial_x^{j+1} (v_m + w_m) + 
a_{(2j)im}(t,x) \partial_x^{j} (v_m + w_m) \Bigr) \\ \times
(\partial_x^j v_i \rho + j \partial_x^{j-1} v_i)\, dx d\tau \\-
2 \sum\limits_{j=0}^{l} \iint_{Q_t} \Bigl(g_{ji}(t,x,u,\dots,\partial_x^{l-1} u) - 
g_{ji}(t,x,\widetilde u,\dots,\partial_x^{l-1} \widetilde u) \Bigr) \\ \times
(\partial_x^j v_i \rho + j \partial_x^{j-1} v_i)\, dx d\tau.
\end{multline}
Note that by virtue of \eqref{1.8} uniformly in $i$ and $x$
\begin{equation}\label{3.20}
(2l+1)a_{(2l+1)i} - 2a_{(2l)i}\rho(x) \geq \alpha_0>0.
\end{equation}
It follows from \eqref{2.1} for $p=2$ that if $j\leq l-1$
\begin{multline}\label{3.21}
\iint_{Q_t} |\partial_x^{j+1} v_m|\cdot |\partial_x^j v_i|\, dx d\tau 
\leq c\int_0^t \left[\|\partial_x^l v\|_{(L_2(I))^n}^{(2l-1)/l)} \|v\|_{(L_2(I))^n}^{1/l} + 
\|v\|_{(L_2(I))^n}^2 \right] \,d\tau \\
\leq \varepsilon \iint_{Q_t} |\partial_x^l v|^2 \,dx d\tau +
c(\varepsilon) \iint_{Q_t} |v|^2\rho \,dx d\tau,
\end{multline}
where $\varepsilon>0$ can be chosen arbitrarily small;
\begin{multline}\label{3.22}
\iint_{Q_t} |\partial_x^{j+1} w_m|\cdot |\partial_x^j v_i|\, dx d\tau 
\leq  \Bigl(\iint_{Q_t} (\partial_x^j v_i)^2\, dx d\tau 
\iint_{Q_t} (\partial_x^{j+1} w_m)^2 \, dx d\tau \Bigr)^{1/2} \\ \leq
\varepsilon \iint_{Q_t} |\partial_x^l v|^2 \,dx d\tau +
c(\varepsilon) \iint_{Q_t} |v|^2\rho \,dx d\tau   + c\|w\|^2_{(X(Q_T))^n}.
\end{multline}
Next, similarly to \eqref{3.7}
\begin{multline}\label{3.23}
|g_{ji}(t,x,u,\dots,\partial_x^{l-1} u) - 
g_{ji}(t,x, \widetilde u,\dots,\partial_x^{l-1} \widetilde u)|  
\\ \leq c\sum\limits_{k=0}^{l-1} \sum\limits_{m=0}^{l-1} \left(|\partial_x^m u |^{b_1(j,k,m)} + |\partial_x^m \widetilde u |^{b_1(j,k,m)} + |\partial_x^m u |^{b_2(j,k,m)}+
|\partial_x^m \widetilde u |^{b_2(j,k,m)} \right) \\ \times |\partial_x^k (v + w)|.
\end{multline}
Note that, for example, for $j\leq l$, $k,m\leq l-1$ if $0\leq b \leq (4l-2j-2k)/(2m+1)$
\begin{multline}\label{3.24}
\int_I |\partial_x^m u|^b |\partial_x^k v|\cdot |\partial_x^j v|\,dx \leq  
\sup\limits_{x\in I} |\partial_x^m u|^b 
\Bigl( \int_I |\partial_x^k v|^2\,dx \int_I |\partial_x^j v|^2\,dx \Bigr)^{1/2} \\ \leq
c\sup\limits_{x\in I} |\partial_x^m u|^b
\Bigl[\Bigl( \int_I |\partial_x^l v|^2\,dx\Bigl)^{(k+j)/(2l)}
\Bigl(\int_I |v|^2\,dx\Bigr)^{(2l-j-k)/(2l)} + \int_I |v|^2\,dx \Bigr] \\ \leq
\varepsilon \int_I |\partial_x^l v|^2\,dx + c(\varepsilon)
\Bigl[\sup\limits_{x\in I} |\partial_x^m u|^{2lb/(2l-j-k)} + 
\sup\limits_{x\in I} |\partial_x^m u|^b \Bigr] \int_I |v|^2\rho\,dx,
\end{multline}
where
\begin{multline}\label{3.25}
\int_0^T \sup\limits_{x\in I} |\partial_x^m u|^{2lb/(2l-j-k)} \,dt \\ \leq
\sup\limits_{t\in (0,T)} \Bigl(\int_I |u|^2\,dx\Bigr)^{(2l-2m-1)b/(4l-2j-2k)} 
\int_0^T \Bigl( \int_I |\partial_x^l u|^2 \,dx\Bigr)^{(2m+1)b/(4l-2j-2k)}\,dt \\ \leq
c(T) \|u\|_{(X(Q_T))^n}^{2lb/(2l-j-k)}\,dt;
\end{multline}
also divide $b$ into two parts: $b =b' + b''$, where $0 \leq b' \leq (2l-2j)/(2m+1)$,
$0\leq b'' \leq (2l-2k)/(2m+1)$, then similarly to \eqref{3.24}
\begin{multline}\label{3.26}
\int_I |\partial_x^m u|^b |\partial_x^k w|\cdot |\partial_x^j v|\,dx \leq  
\sup\limits_{x\in I} |\partial_x^m u|^{b' +b''} 
\Bigl(\int_I |\partial_x^j v|^2\,dx  \int_I |\partial_x^k w|^2\,dx\Bigr)^{1/2} \\ \leq
\varepsilon \int_I |\partial_x^l v|^2\,dx + c(\varepsilon)
\Bigl[\sup\limits_{x\in I} |\partial_x^m u|^{2lb'/(l-j)} + 
\sup\limits_{x\in I} |\partial_x^m u|^{2b'} \Bigr] \int_I |v|^2 \rho\,dx \\ +
c\int_I |\partial_x^l w|^2\,dx + 
c\Bigl[\sup\limits_{x\in I} |\partial_x^m u|^{2lb''/(l-k)} + 
\sup\limits_{x\in I} |\partial_x^m u|^{2b''} \Bigr] \int_I |w|^2\,dx,
\end{multline}
where similarly to \eqref{3.25}
\begin{equation}\label{3.27}
\int_0^T \sup\limits_{x\in I} |\partial_x^m u|^{2lb'/(l-j)} \,dt,
\int_0^T \sup\limits_{x\in I} |\partial_x^m u|^{2lb''/(l-k)} \,dt \leq c(T,K).
\end{equation}
Gathering \eqref{3.20}--\eqref{3.27} we deduce from inequality \eqref{3.19} that  
\begin{multline}\label{3.28}
\int_I v_i^2(t,x)\rho\,dx + \alpha_0 \iint_{Q_t} (\partial_x^l v_i)^2 \,dx d\tau \leq
\frac{\alpha_0}{2n} \iint_{Q_t}|\partial_x^l v|^2 \,dx d\tau \\+ 
\int_0^t \gamma(\tau) \int_I |v|^2\rho\,dx d\tau + 
2 \int_{0}^t \|f - \widetilde f\|_{(L_2(I))^n} \|v_i\|_{L_2(I)}  \, d\tau + c(T,K)\|w\|^2_{(X(Q_T))^n},
\end{multline}
where $\|\gamma\|_{L_1(0,T)} \leq c(T,K)$. Summing inequalities \eqref{3.28} with respect to $i$, using estimate \eqref{3.17} and applying Gronwall lemma we complete the proof.
\end{proof}

In this section it remains to prove Theorem~\ref{T1.2}.

\begin{proof}[Proof of Theorem~\ref{T1.2}]
In general, the proof repeats the proof of the existence part of Theorem~\ref{T1.1}. The desired solution is constructed as a fixed point of the map $\Theta$ from \eqref{3.1}. In comparison with \eqref{3.3} here we obtain the following estimate: let
\begin{equation}\label{3.29}
\sigma = \frac{\min\limits_{j,k,m} (4l-2j-2k - (2m+1)b_2(j,k,m))}{4l}
\end{equation}
(note that $\sigma>0$ because of \eqref{1.16}), then
\begin{multline}\label{3.30}
\|g_{j}(t,x,v,\dots,\partial_x^{l-1} v)\|_{(L_{2l/(2l-j)}(0,T;L_2(I)))^n}  \\ \leq 
c(T) T^{\sigma} \sum\limits_{k=0}^{l-1} \sum\limits_{m=0}^{l-1} \sum\limits_{i=1}^2  \|v\|_{(X(Q_T))^n}^{b_i(j,k,m)+1}.
\end{multline}
and similarly to \eqref{3.6}, \eqref{3.9}
\begin{equation}\label{3.31}
\|\Theta v\|_{(X(Q_T))^n} \leq c(T)c_0 + c(T)T^\sigma 
\left( \|v\|_{(X_(Q_T))^n}^{b_1+1} + \|v\|_{(X_(Q_T))^n}^{b_2+1}\right).
\end{equation}
\begin{multline}\label{3.32}
\|\Theta v_1 - \Theta v_2\|_{(X(Q_T))^n}  \\ \leq 
c(T)T^\sigma \left( \|v_1\|_{(X(Q_T))^n}^{b_1} + \|v_2\|_{(X(Q_T))^n}^{b_1} + 
\|v_1\|_{(X(Q_T))^n}^{b_2} + \|v_2\|_{(X(Q_T))^n}^{b_2}\right)  \\ \times
\|v_1-v_2\|_{(X(Q_T))^n}.
\end{multline}

Now for the fixed $\delta$ choose $T_0>0$ such that
\begin{equation}\label{3.33}
4c(T_0) T_0^\sigma \left( (2c(T_0) \delta)^{b_1} + (2c(T_0)\delta)^{b_2}\right)  \leq 1
\end{equation}
(it is possible since $c(T)$ does not decrease in $T$) and then for every $T\in (0,T_0]$ choose an arbitrary $r$ such that
\begin{equation}\label{3.34}
r\geq 2c(T) \delta, \quad 4c(T)T^\sigma (r^{b_1} + r^{b_2}) \leq 1
\end{equation}
(this set is not empty because of \eqref{3.33}). Then the map $\Theta$ is a contraction on the ball $\overline X_{rn}(Q_T)$.

In order to prove uniqueness in the whole space note that for an arbitrary large $r$ the value of $T_0$ can be chosen sufficiently small such that the solution of the considered problem $u\in (X(Q_{T_0})^n)$  is the unique fixed point of the contraction $\Theta$ in $\overline X_{rn}(Q_{T_0})$.
\end{proof}

\section{The inverse problem}\label{S4}

 We start with the linear case. The following lemma is the crucial part of the study. 
 
\begin{lemma}\label{L4.1}
Let the assumptions on the functions $a_j$ from the hypothesis of Theorem~\ref{T1.3} be satisfied. Let condition \eqref{1.6} be valid and for any $i=1,\dots,n$, satisfying $m_i>0$, for $k=1,\dots,m_i$ the functions $\omega_{ki}$ satisfy condition \eqref{1.12}, $\varphi_{ki} \in \widetilde W_1^1(0,T)$, $h_{ki} \in C([0,T];L_2(I))$ and for the corresponding functions $\psi_{kji}$ conditions \eqref{1.19} be verified. Then there exists a unique set of $M$ functions $F=\{F_{ki}(t), i: m_i>0, k=1,\dots,m_i\} = \Gamma\{\varphi_{ki}, i: m_i>0, k=1,\dots,m_i\} \in (L_1(0,T))^M$ such that for $f = (f_1,\dots,f_n)^T \equiv H F$, where for any $i=1,\dots,n$ the function $f_i(t,x)$ is presented by formula \eqref{1.4}, where $h_{0i} \equiv 0$ ($f_i \equiv 0$ if $m_i=0$), the corresponding function 
\begin{equation}\label{4.1}
u = S_0 f = (S_0 \circ H) F,
\end{equation}
verifies all conditions \eqref{1.5}. Moreover, the linear operator $\Gamma: \bigl(\widetilde W_1^1(0,T)\bigr)^M \to \bigl(L_1(0,T)\bigr)^M$ is bounded and its norm does not decrease in $T$.
\end{lemma}

\begin{proof}
First of all note that by virtue of \eqref{1.18}, \eqref{1.19}
\begin{equation}\label{4.2}
|\Delta_i(t)| \geq \Delta_0 >0, \ |\psi_{kji}(t)| \leq \psi_0, \quad t\in [0,T].
\end{equation}

On the space $\bigl(L_1(0,T)\bigr)^M$ introduce $M$ linear operators $\Lambda_{ki} = Q(\omega_{ki})\circ S_0\circ H$. Let $\Lambda = \{\Lambda_{ki}\}$. Then since $H F \in \bigl(L_1(0,T;L_2(I))\bigr)^n$ by Theorem \ref{T2.1} and Lemma \ref{L2.1} the operator $\Lambda$ acts from the space $\bigl(L_1(0,T)\bigr)^M$ into the space $\bigl(\widetilde W_1^1 (0,T)\bigr)^M$ and is bounded.

Note that the set of equalities $\varphi_{ki} = \Lambda_{ki} F$, $i: m_i>0, k=1,\dots,m_i$, for $F\in \bigl(L_1(0,T)\bigr)^M$ obviously means that the set of functions $F$ is the desired one. 

Let for $i$ verifying $m_i>0$
\begin{multline}\label{4.3}
\widetilde r(t;u_i,\omega_{ki}) \equiv 
(-1)^{l+1} \int_I u_i(t,x) \left(a_{(2l+1)i}\omega_{ki}^{(2l+1)} - a_{2l} \omega_{ki}^{(2l)} \right) \,dx \\+
\sum\limits_{m=1}^n \sum\limits_{j=0}^{l-1} (-1)^{j+1} \int_I u_m(t,x) \left[ (a_{(2j+1)im} \omega_{ki}^{(j)})^{(j+1)} - (a_{(2j)im} \omega_{ki}^{(j)})^{(j)} \right] \,dx,
\end{multline}
where $u = (u_1,\dots,u_n)^T = (S_0\circ H)F$. Then from \eqref{2.17} it follows that for
$q(t;u_i,\omega_{ki}) = \bigl(\Lambda_{ki}F\bigr)(t)$
\begin{equation}\label{4.4}
q'(t;u_i,\omega_{ki}) = \widetilde r(t;u_i,\omega_{ki}) + 
\sum\limits_{j=1}^{m_i} F_{ji}(t) \psi_{kji}(t),
\end{equation}
where the functions $\psi_{kji}$ are given by formula \eqref{1.18}. Let
\begin{equation}\label{4.5}
y_{ki}(t) \equiv q'(t;u_i,\omega_{ki}) - \widetilde r(t;u_i,\omega_{ki}),\quad k=1,\dots,m_i.
\end{equation}
and $\widetilde\Delta_{ki}(t)$ be the determinant of the $m_i\times m_i$-matrix, where in comparison with the matrix $\bigl(\psi_{kji}(t)\bigr)$ the $k$-th column is substituted by the column $\bigl(y_{1i}(t),\dots,y_{m_ii}(t)\bigr)^T$. Then \eqref{4.4} implies
\begin{equation}\label{4.6}
F_{ki}(t) = \frac{\widetilde\Delta_{ki}(t)}{\Delta_i(t)},\quad k=1,\dots,m_i.
\end{equation}
Let
\begin{equation}\label{4.7}
z_{ki}(t) \equiv \varphi'_{ki}(t) - \widetilde r(t;u_i,\omega_{ki}),\quad k=1,\dots,m_i,
\end{equation}
and $\Delta_{ki}(t)$ be the determinant of the $m_i\times m_i$-matrix, where in comparison with $\widetilde\Delta_{ki}(t)$ the $k$-th column $\bigl(y_{1i}(t),\dots,y_{m_ii}(t)\bigr)^T$ is substituted by the column $\bigl(z_{1i}(t),\dots,z_{m_ii}(t)\bigr)^T$.

Introduce operators $A_{ki}: L_1(0,T) \to L_1(0,T)$ by
\begin{equation}\label{4.8}
(A_{ki}F)(t) \equiv \frac{\Delta_{ki}(t)}{\Delta_i(t)}
\end{equation}
and let $AF = \{A_{ki}F\}$, $A: \bigl(L_1(0,T)\bigr)^M \to \bigl(L_1(0,T)\bigr)^M$.

Note that $\varphi_{ki} = \Lambda_{ki} F$, for all $i: m_i>0, k=1,\dots,m_i$ if and only if $AF=F$.

Indeed, if $\varphi_{ki} = \Lambda_{ki} F$, then $\varphi'_{ki}(t) \equiv q'(t;u_i,\omega_{ki})$ for the function $q(t;u_i,\omega_{ki}) \equiv \bigl(\Lambda_{ki}F\bigr)(t)$  and equalities \eqref{4.5}, \eqref{4.7} yield $\Delta_{ki}(t) \equiv \widetilde\Delta_{ki}(t)$. Hence, $AF=F$. 

Vice versa, if $AF=F$, then $\Delta_{ki}(t) \equiv \widetilde\Delta_{ki}(t)$ and the condition $\Delta_i(t) \ne 0$ implies $z_{ki}(t) \equiv y_{ki}(t)$ and so $\varphi'_{ki}(t) \equiv q'(t;u_i,\omega_{ki})$. Since $\varphi_{ki}(0) = q(0;u_i,\omega_{ki})=0$, we have $q(t;u_i,\omega_{ki}) \equiv \varphi_{ki}(t)$.

Next, we show that the operator $A$ is a contraction under the choice of a special norm in the space $\bigl(L_1(0,T)\bigr)^M$. 

Let $F_1, F_2 \in \bigl(L_1(0,T)\bigr)^M$, $u_m \equiv  (S_0 \circ H)F_m$, $m=1,\; 2$, and let $\Delta_{ki}^*(t)$ be the determinant of the $m_i\times m_i$-matrix, where in comparison with the matrix $\bigl(\psi_{kji}(t)\bigr)$ the $k$-th column is substituted by the column, where on the $j$-th line stands the element
$\widetilde r(t;u_{1i},\omega_{ji}) - \widetilde r(t;u_{2i},\omega_{ji}) = 
\widetilde r(t;u_{1i} - u_{2i},\omega_{ji})$. Then
\begin{equation}\label{4.9}
\bigl(A_{ki}F_1\bigr)(t) - \bigl(A_{ki}F_2\bigr)(t) = - \frac{\Delta^*_{ki}(t)}{\Delta_i(t)}.
\end{equation}
By \eqref{2.8} for $t\in [0,T]$ 
\begin{equation}\label{4.10}
\|u_{1}(t,\cdot) - u_{2}(t,\cdot)\|_{(L_2(I))^n} \leq 
c(T) \sum\limits_{i: m_i>0} \sum\limits_{j=1}^{m_i} \|h_{ji}\|_{C([0,T];L_2(I))} \|F_{1ji} - F_{2ji}\|_{L_1(0,t)}.
\end{equation}
Let $\gamma>0$, then by virtue of \eqref{4.2}, \eqref{4.3}, \eqref{4.9} and \eqref{4.10}
\begin{multline}\label{4.11}
\|e^{-\gamma t} (AF_1 - AF_2)\|_{(L_1(0,T))^M}  \\ \leq 
\frac{c\bigl(\{\|\omega_{ji}\|_{H^{2l+1}(I)}\},\psi_0\bigr)}{\Delta_0} 
\int_0^T e^{-\gamma t} \|u_{1}(t,\cdot) - u_{2}(t,\cdot)\|_{(L_2(I))^n} \,dt \\ \leq
c\bigl(T,\bigl(\{\|\omega_{ji}\|_{H^{2l+1}(I)}\},\psi_0,\{\|h_{ji}\|_{C([0,T];L_2(I))}\}\bigr) \\ \times \int_0^T e^{-\gamma t} \int_0^t \sum\limits_{i: m_i>0} \sum\limits_{j=1}^{m_i}
|F_{1ji}(\tau) - F_{2ji}(\tau)| \,d\tau dt \\ = 
c\int_0^T \sum\limits_{i: m_i>0} \sum\limits_{j=1}^{m_i} 
|F_{1ji}(\tau) - F_{2ji}(\tau)| \int_\tau^T e^{-\gamma t} \,dt d\tau \\ \leq
\frac{c}{\gamma}\| e^{-\gamma\tau} (F_1 - F_2)\|_{(L_1(0,T))^M}.
\end{multline}
It remains to choose respectively large $\gamma$.

As a result, for any set of functions $\varphi_{ki} \in \bigl(\widetilde W_1^1(0,T)\bigr)^M$ there exists a unique set of functions $F \in \bigl(L_1(0,T)\bigr)^M$ verifying $AF=F$, that is $\varphi_{ki} = \Lambda_{ki} F$. It means that the operator $\Lambda$ is invertible and so the Banach theorem implies that the inverse operator $\Gamma= \Lambda^{-1} : \bigl(\widetilde W_1^1(0,T)\bigr)^M \to \bigl(L_1(0,T)\bigr)^M$ is continuous. In particular,
\begin{equation}\label{4.12}
\|\Gamma\{\varphi_{ki}\}\|_{(L_1(0,T))^M} \leq c(T) \|\{\varphi_{ki}\}\|_{(\widetilde W_1^1(0,T))^M}.
\end{equation}

For an arbitrary $T_1>T$ extend the functions $\varphi_{ki}$ by the constant $\varphi_{ki}(T)$ to the interval $(T,T_1)$. Then the analog of inequality \eqref{4.12} on the interval $(0,T_1)$ for such a function evidently holds with $c(T) \leq c(T_1)$. It means that the norm of the operator $\Gamma$ is non-decreasing in $T$.
\end{proof}

The next result is the solution of the corresponding inverse problem for the full linear problem.

\begin{theorem}\label{T4.1}
Let the function $f$ be given by formula \eqref{1.4} and condition \eqref{1.6} be satisfied.
Let the functions $a_i$, $u_0$, $(\mu_0,\dots,\mu_{l-1})$, $(\nu_0,\dots,\nu_l)$, $h_0$, $\varphi_{ki}$, $\omega_{ki}$, $h_{ki}$ satisfy the hypothesis of Theorem~\ref{T1.3} and the functions $G_j$ satisfy the hypothesis of Theorem \ref{T2.1}. Then there exists a unique set of $M$ functions $F=\{F_{ki}(t), i: m_i>0, k=1,\dots,m_i\} \in (L_1(0,T))^M$ such that the corresponding unique solution $u\in \bigl(X(Q_T)\bigr)^n$ of problem \eqref{2.3}, \eqref{1.2},  \eqref{1.3} verifies all conditions \eqref{1.5}. Moreover, the functions $F$ and $u$ are given by formulas
\begin{gather}\label{4.13}
F = \Gamma \Bigl\{\varphi_{ki} - Q(\omega_{ki})\bigl( \widetilde S W + S_0 h_0 + \sum\limits_{j=0}^l \widetilde S_j G_j\bigr)_i \Bigr\}, \\
\label{4.14}
u = \widetilde S W + S_0 h_0 + \sum\limits_{j=0}^l S_j G_j +(S_0\circ H) F.
\end{gather}
\end{theorem}

\begin{proof}
Set
$$
v \equiv S(u_0, (\mu_0,\dots,\mu_{l-1}), (\nu_0,\dots,\nu_l), h_0, (G_0,\dots,G_l)) = \widetilde S W + S_0 h_0  +\sum\limits_{j=0}^l \widetilde S_j G_j.
$$
Lemma \ref{L2.1} implies $Q(\omega_{ki}) v_i \in W_1^1(0,T)$. Moreover, by virtue of \eqref{1.17} $Q(\omega_{ki}) v_i \big|_{t=0} = \varphi_{ki}(0)$. Set
$$
\widetilde\varphi_{ki} \equiv \varphi_{ki} - Q(\omega_{ki}) v_i,
$$
then $\widetilde\varphi_{ki} \in \widetilde W_1^1(0,T)$. In turn, Lemma \ref{L4.1} implies that the functions $F\equiv \Gamma\{\widetilde\varphi_{ki}\}$ and $u \equiv v +(S_0\circ H )F$
provide the desired result. Uniqueness also follows from Lemma \ref{L4.1}.
\end{proof}

Now we pass to the nonlinear equation.

\begin{proof}[Proof of Theorem \ref{T1.3}]
On the space $\bigl(X(Q_T)\bigr)^n$ consider a map $\Theta$
\begin{gather}\label{4.15}
u = \Theta v \equiv \widetilde S W +S_0 h_0 - 
\sum\limits_{j=0}^l \widetilde S_j g_j(t,x,v,\dots,\partial_x^{l-1} v) +(S_0\circ H)F, \\
\label{4.16}
F \equiv \Gamma \Bigl\{ \varphi_{ki} - Q(\omega_{ki})\bigl( \widetilde S W +S_0 h_0 - \sum\limits_{j=0}^l \widetilde S_j g_j(t,x,v,\dots,\partial_x^{l-1} v)\bigr)_i\Bigr\}.
\end{gather}
Then estimate \eqref{3.5} and Theorem~\ref{T4.1} applied to $G_j(t,x) \equiv 
g_j(t,x,v,\dots,\partial_x^{l-1} v)$ ensure that the map $\Theta$ exists. 

Apply Lemmas \ref{L2.3} and \ref{L4.1}, then the function $F$ from \eqref{4.16} is estimated as follows:
\begin{multline}\label{4.17}
\|F\|_{(L_1(0,T))^M} \leq c(T) \Bigl[ \|u_0\|_{(L_2(I))^n} + \|(\mu_0,\dots,\mu_{l-1})\|_{(\mathcal B^{l-1}(0,T))^n} \\ + 
\|(\nu_0,\dots,\nu_l)\|_{(\mathcal B^l(0,T))^n} +
\|h_0\|_{(L_1(0,T;L_2(I)))^n}+ \|\{\varphi'_{ki}\}\|_{(L_1(0,T))^M}  \\+
\|v\|_{(X(Q_T))^n}^{b_1+1} + \|v\|_{(X(Q_T))^n}^{b_2+1} \Bigr];
\end{multline} 
therefore, since also
$$
\|HF\|_{(L_1(0,T;L_2(I)))^n} \leq \max\limits_{i: m_i>0, k=1,\dots,m_i} 
\bigl(\|h_{ki}\|_{C([0,T];L_2(I))}\bigr) \|F\|_{(L_1(0,T))^M},
$$
Theorem \ref{T2.1} provides for the map $\Theta$ estimate \eqref{3.6}.

Next, for any functions $v_1, v_2 \in \bigl(X(Q_T)\bigr)^n$ since 
\begin{multline}\label{4.18}
\Theta v_1 - \Theta v_2 = -\sum\limits_{j=0}^l \widetilde S_j \left[g_j(t,x,v_1,\dots,\partial_x^{l-1} v_1) - g_j(t,x,v_2,\dots,\partial_x^{l-1} v_2)\right] \\+
(S_0\circ H \circ\Gamma )\Bigl\{Q(\omega_{ki})\Bigl(\!\sum\limits_{j=0}^l \widetilde S_j \left[g_j(t,x,v_1,\dots,\partial_x^{l-1} v_1) - g_j(t,x,v_2,\dots,\partial_x^{l-1} v_2)\right]\Bigr)_i\Bigr\},
\end{multline}
using \eqref{3.8} we derive estimate \eqref{3.9}. 

Now choose $r>0$ and $\delta>0$ as in \eqref{3.10}, \eqref{3.11}.
Then it follows from \eqref{3.6} and \eqref{3.9} that on the ball $\overline X_{rn}(Q_T)$ the map $\Theta$ is a contraction. Its unique fixed point $u\in \bigl(X(Q_T)\bigr)^n$ is the desired solution. Moreover, Theorem \ref{T4.1} implies that the function $F$ in \eqref{4.16} (for $v\equiv u$) is determined in a unique way. 

Continuous dependence is obtained similarly to \eqref{3.6}, \eqref{3.9}.
\end{proof}

\begin{proof}[Proof of Theorem \ref{T1.4}]
In general, the proof repeats the previous argument. The desired solution is constructed as a fixed point of the map $\Theta$ from \eqref{4.15}, \eqref{4.16}. In comparison with \eqref{3.6}, \eqref{3.9} here (also with the use of \eqref{4.18}) we obtain estimates \eqref{3.31} and \eqref{3.32}, where $\varepsilon$ is defined in \eqref{3.29}.

The end of the proof is the same as in Theorem~\ref{T1.2} (with the corresponding supplements as in Theorem~\ref{T1.3}).
\end{proof}

\section*{Acknowledgments}

This paper has been supported by Russian Science Foundation grant 23-21-00101.

\end{document}